\documentclass[12pt]{article}
\usepackage {graphics} 
\usepackage {graphicx}  
\usepackage {pstricks,pst-node,  pst-coil} 
\onecolumn  
\usepackage{amsmath,amssymb,latexsym}
\usepackage{amsmath,amsthm}
\newtheorem{tw}{Theorem}  
\newtheorem{lem}[tw]{Lemma}

\newcommand{\pusty}{\emptyset}   
\newcommand{\trou}{\vspace{.5cm}}

\newcommand{\cbdo}{\hfill \rule{.1in}{.1in}}

\begin{document}
\hyphenation{every}
\title{A note on uniquely embeddable $2$-factors}
\author{Igor Grzelec\thanks{The corresponding author. Email:  grzelec@agh.edu.pl}, Monika Pilśniak, Mariusz Woźniak\\Department of Discrete Mathematics\\AGH University of Krakow\\Poland}
\maketitle
\begin{abstract}
Let $C_{n_1}\cup C_{n_2}\cup \ldots \cup C_{n_k}$ be a 2-factor \emph{i.e.} a vertex-disjoint union of cycles. In this note we completely characterize those 2-factors that are uniquely embeddeble in their complement.
\end{abstract}

\section{Introduction}
We consider only finite, undirected graphs of order $n=|V(G)|$ and size $e(G)=|E(G)|$. All graphs will be assumed to have neither loops nor multiple edges.

We shall need some additional definitions in order to formulate the results. If a graph $G$ has order $n$ and size $m$, we say that $G$ is an $(n,m)$ graph.

Assume now that $G_1$ and $G_2$ are two graphs with disjoint vertex sets. The {\it union} $G=G_1\cup G_2$  has $V(G)=V(G_1)\cup V(G_2)$ and  $E(G)=E(G_1)\cup E(G_2)$. If a graph is the union of $n$ ($\ge 2$)  disjoint copies of a graph $H$, then we write $G=nH$.

For our next  operation, the conditions are quite different. Let now  $G_1$ and $G_2$ be  graphs with $V(G_1)=V(G_2)$ and $E(G_1)\cap E(G_2)=\pusty$. The {\it edge sum} $G_1\oplus G_2$ has $V(G)=V(G_1)=V(G_2)$ and $E(G)=E(G_1)\cup E(G_2)$.

An {\it embedding} of $G$ (in its complement $\overline{G}$) is a permutation $\sigma$ on $V(G)$ such that if an edge $xy$ belongs to $E(G)$, then $\sigma (x)\sigma (y)$ does not belong to $E(G)$.

In others words, an embedding is an (edge-disjoint) 
{\it placement} (or {\it packing}) of two copies of $G$ into a complete graph $K_n$.

 The following theorem was proved, independently, in \cite{BoEl}, \cite{BuSc} and \cite{SaSp}.

\begin{tw} \label{E2}
Let $G=(V,E)$ be a graph of order $n$. If $|E(G)|\leq n-2$ then $G$ can be embedded in its complement $\overline{G}$.  \cbdo
\end{tw}

The example of the star $K_{1,n-1}$  shows  that  Theorem \ref{E2} cannot be improved by raising the size of $G$. The following theorem, proved in \cite{BuSc1}, give the full characterization of graphs with order $n$ and size $n-1$ that are embeddable.   

\begin{tw} \label{E3}
Let $G=(V,E)$ be a graph of order $n$. If $|E(G)|\leq n-1$ then either $G$ is embeddable or $G$ is isomorphic to one of the following graphs: $K_{1, n-1}$,  $K_{1, n-4} \cup K_3$ with $n\geq8$, $K_1 \cup K_3$, $K_2 \cup K_3$, $K_1 \cup 2K_3$, $K_1 \cup C_4$.\cbdo
\end{tw}

Let us consider now the  problem of the uniqueness. First, we have to precise what we mean by {\it distinct} embeddings.

Let $\sigma$ be an embedding of the graph $G=(V,E)$. We denote by  $\sigma (G)$ the graph with the vertex set $V$ and the edge set   $\sigma ^* (E)$ where the map $\sigma ^*$ is induced by  $\sigma$. Since, by definition of an embedding, the sets $E$ and $\sigma ^* (E)$ are disjoint we may form the graph $G\oplus\sigma (G)$. 

Two embeddings $\sigma_1$, $\sigma_2$ of a graph $G$ are said to be {\it distinct} if the graphs $G\oplus\sigma_1 (G)$ and $G\oplus\sigma_2 (G)$ are not isomorphic. A graph $G$  is called {\it uniquely embeddable} if for all embeddings $\sigma$ of $G$, all graphs $G\oplus\sigma (G)$ are isomorphic.

The problem of uniqueness has so far been the subject of two papers. The next theorem, proved in \cite{Woz1}, characterizes all $(n,n-2)$ graphs that are uniquely embeddable.

\begin{tw} \label{e1}
Let $G$ be a graph of order $n$ and size $e(G)=n-2$. Then either $G$ is not uniquely embeddable or $G$ is isomorphic to one of the six following graphs: $K_2\cup K_1$, $2K_2$, $K_3\cup 2K_1$, $K_3\cup K_2\cup K_1$, $K_3\cup 2K_2$, $2K_3\cup 2K_1$.   \cbdo
\end{tw}

By \textit{double star} $S(p,q)$ we mean a tree obtained from two stars $S_{p+1}$ and $S_{q+1}$ by joining their centers by an edge. By $S'_n$ for $n=q+3$ we denote a double star $S(1,q)$. In \cite{Otwinowska} the following characterization of uniquely embeddable forests was proved.

\begin{tw} 
Let $F$ be a forest of order $n$ having at least one edge. Then either $F$ is not uniquely embeddable or $F$ is isomorphic to one of the following graphs: $K_2 \cup K_1$, $2K_2$, $3K_3$, $S(p,q)$ or $S'_n$. \cbdo
\end{tw}

{\bf Remark.} The main references of the paper and of other packing problems are the  following survey papers
\cite{Yap}, \cite{Woz5} or \cite{Woz6}.

The aim of this note is to consider the problem of uniqueness of embedding for 2-factors. 

\begin{tw}\label{Sums}
Let $G=C_{n_1}\cup C_{n_2}\cup \ldots \cup C_{n_k}$ be a vertex-disjoint union of $k$ cycles. The cycles $C_3$, $C_4$ and the graph $2C_3$ are not embeddable. The graphs: $C_5$, $C_6$, $C_3 \cup C_4$, $C_3 \cup C_5$, $3C_3$ and $4C_3$ are uniquely embeddable. For other graph $G$ there exist at least two distinct embeddings.
\end{tw}

The proof of Theorem \ref{Sums} is given in the next sections. Section 2 contains the case of cycles ($k=1$), while Section 3 and Section 4 deal with union of two or three cycles and union of $k$ cycles where $k\geq4$ respectively. 

{\bf Remark.} We will often present packing of a graph $G$ in figures. Therefore we need to introduce additional notation. We say that the first (initial) copy of a graph $G$ is \textit{black} and the second copy of $G$ is \textit{red} in the packing of $G$. In figures we draw with a continuous black line the first copy of $G$ and we draw with a dashed red line the second copy of $G$.

We denote by $B(X, Y)$ graph which contains all edges between two disjoint sets $X$ and $Y$. We will often use this notation for cycles $C_4$ in packings which contain these cycles. The following lemma will be useful through the proof.
\begin{lem}
If a graph $G = C_{n_1}\cup C_{n_2}\cup \ldots \cup C_{n_k}$ has a packing $\sigma$ such that the graph $G\oplus\sigma (G)$ is not connected $($a disconnected packing$)$, then $G$ has another packing $\hat{\sigma}$ such that the graph $G\oplus\hat{\sigma} (G)$ is connected $($a connected packing$)$. In particular, the graph $G$ is not uniquely packable.
\label{l1}
\end{lem} 
\begin{proof} 
Let's choose the packing $\sigma$ with the least number of connected components. If $H=G\oplus \sigma(G)$ is connected, we're done. If not, let $H_1, H_2$ be two components of $H$.

Suppose $y_1$ is a vertex of $H_1$ such that removing the two red edges $y_1^-y_1$ and $y_1^+y_1$ leaves
$H_1$ connected where $y_1^-$ and $y_1^+$ are neighbors of $y_1$ on the red cycle in $H_1$. In the same way, we select $y_2$, a vertex belonging to the component $H_2$.

If now instead of the edges $y_1^-y_1$ and $y_1^+y_1$ we draw two red edges $y_2^-y_1$ and $y_2^+y_1$, and instead of the edges $y_2^-y_2$ and $y_2^+y_2 $ we draw two red edges $y_1^-y_2$ and $y_1^+y_2$, we get a new packing $\hat{\sigma} $ where two components $H_1$ and $H_2$ become one connected component, contradiction with the choice of \linebreak packing $\sigma$.

To complete the proof, it suffices to show that for each connected component of the graph $H$ one can choose a vertex as we did above. 

Just take a vertex that is not cut vertex. Such a vertex exists in every connected component. For example, the last vertex on the longest component path.
\end{proof}

\section{Case $k=1$}
Let $G=C_n$ be a cycle of order $n$. It is easy to see that neither $C_3$ nor $C_4$ is embeddable. 

The cycle $C_5$ is embeddable but for each embedding $\sigma$ we have $C_5\oplus \sigma (C_5) = K_5$. So, $C_5$ is uniquely embeddable.

The cycle $C_6$ is also embeddable. For each embedding $\sigma$ the graph  $C_6\oplus \sigma (C_6)$ is a 4-regular subgraph of $K_6$. The complement of such a graph is a 1-factor in $K_6$. Thus, all these graphs are isomorphic. So, $C_6$ is uniquely embeddable.

Two distinct embeddings of $C_7$ are given in Figure \ref{fig1}. In the first one, the complement of the graph $C_n\oplus \sigma (C_n)$ is isomorphic to $C_7$ while in the second one, to $C_3 \cup C_4$. 
\begin{figure}[h!]
\centering
\includegraphics[width=9cm]{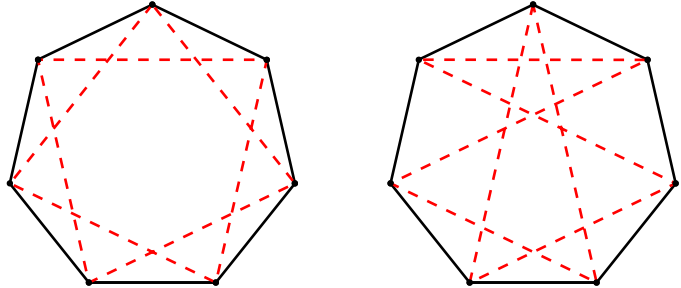}
\caption{Two distinct embeddings of $C_7$.}
\label{fig1}
\end{figure}

For $n\ge 8$ we shall show that there are at least two distinct embeddings of $C_n$: 

A) One such that the graph $C_n\oplus \sigma (C_n)$ contains  a clique $K_4$ and 

B) another such that the graph $C_n\oplus \sigma (C_n)$ is $K_4$-free.

\trou
{\bf Case A.} 

Denote by $x, a_1, a_2, a_3, a_4, y$ six consecutive vertices of $G=C_n$ and by $P$ the path joining $x$ and $y$ obtained from $C_n$ by removing vertices $\{a_1, a_2, a_3, a_4\}$ \emph{i.e.} $P = G'=G\setminus \{a_1, a_2, a_3, a_4\}$. Since $n\ge 8$ $P$, has at least four vertices. By Theorem~\ref{E3} there is a permutation, say $\sigma '$ being an embedding of $P$. Let $x'= \sigma ' (x)$ and $y'= \sigma ' (y)$. Figure~\ref{fig2} shows how to extend $\sigma '$ to get an embedding of $C_n$. Let us observe that the vertices $\{a_1, a_2, a_3, a_4\}$ induce a clique $K_4$. 
\begin{figure}[h!]
\centering
\includegraphics[width=10cm]{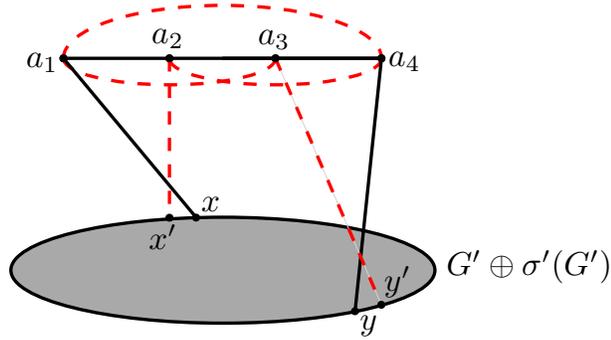}
\caption{If $G'$ is embeddable, then $G \oplus \sigma(G)$ contains $K_4$.}
\label{fig2}
\end{figure}

We will use the above reasoning often, so it will be convenient to formulate it in the form of a lemma.
\begin{lem}
Denote by $a_1, a_2, a_3, a_4$ four vertices of $G$ inducing a path in $G$. If the graph obtained from $G$ by removing vertices $\{a_1, a_2, a_3, a_4\}$ is packable, then the graph $G$ is also packable and there is a packing $\sigma$ of $G$ such that $G \oplus \sigma(G)$ contains a clique $K_4$.
\label{l7}
\end{lem}

\trou
{\bf Case B.} Denote by $v_1, v_2, v_3, \ldots, v_n$  consecutive vertices of $C_n$. We shall consider two cases. 

{\bf Subcase B1.} $n$ is odd. 

Then, the edges $v_iv_{i+2}$ $\pmod n$ define a cycle of length $n$. This cycle can be considered as an image of $C_n$ by a permutation, say $\sigma$. We shall show that  the graph $H= C_n\oplus \sigma (C_n)$ is $K_4$-free. Suppose that $H$ contains a clique on four vertices. It has six edges and it is easy to see that three of them should belong to the first copy of $C_n$ and the remaining three to the second copy of $C_n$, each of these triples forming a path of length three in the corresponding copy. But a path of length three in $C_n$ should be induced by four consecutive vertices $v_i, v_{i+1}, v_{i+2}, v_{i+3}$ $\pmod n$. The fact that $v_i, v_{i+3}$ is not an edge of the second (dashed) copy of $C_n$ finishes the proof of this case. 

{\bf Subcase B2.} $n$ is even. 

It is easy to see that the edges of the form $v_iv_{i+r} \pmod n$ define a cycle of length $n$ if $r$ and $n$ are coprime.  In order to prove the existence of such an integer $r$ we can use, for instance,  the well-known Chebyshev's theorem saying that for each integer $k\ge 4$ there is a prime number between  $k$ and $2k-2$.  Denote by $p$ such a number where $k=\frac{n}{2}$. Since a prime number $p$ and $n$ are surely coprime, $r$ and $n$ where $r=n-p$ are also coprime. Moreover, we have  $3\le r \le \frac{n}{2}-1$.  Similarly as above, it is easy to see that the graph formed by $C_n$ and the edges of the form  $v_iv_{i+r} \pmod n$ is $K_4$-free. 

\section{Case $k=2$ or $k=3$}

\subsection{Case k=2.}
Let $G=C_{n_1} \cup C_{n_2}$, where $n_1\leq n_2$. If $n_1 \geq 5$ we have unconnected packing of $G$ which consists of two components. Each of these components we obtain as a packing of two copies of a cycle in appropriate complete graph from the Case $k=1$. Thus from Lemma \ref{l1} the graph $G$ is not uniquely packable. Therefore we have to consider two subcases $G=C_3 \cup C_p$ where $p \geq 3$ and $G=C_4 \cup C_p$ where $p \geq 4$.

\subsubsection{Subcase $G=C_3 \cup C_p$ where $p \geq 3$.} It is easy to see that $C_3 \cup C_3$ is not embeddable. We show that $C_3 \cup C_4$ and $C_3 \cup C_5$ are uniquely embeddable. We start packing of $\mathbf{C_3 \cup C_4}$ from the black copy in which we denote by $u_1$, $u_2$, $u_3$ consecutive vertices of $C_3$ and by $y_1$, $y_2$, $x_2$, $x_1$ consecutive vertices of $C_4$. To draw a red triangle, we need to use one of the vertices of the black triangle and two opposite vertices from the black cycle $C_4$. Without loss of generality we can use vertices $u_1$, $x_1$, $y_2$, since all possibilities give isomorphic graphs. Now, we have only one possibility to draw the remaining red cycle $C_4$.

We start packing of $\mathbf{C_3 \cup C_5}$ from the black copy in which we denote by $u_1$, $u_2$, $u_3$ consecutive vertices of $C_3$ and by $y_1$, $y_2$, $z$, $x_2$, $x_1$ consecutive vertices of $C_5$. To draw a red triangle, we need to use one of the vertices of the black triangle and two opposite vertices from the black cycle $C_5$. Without loss of generality we can use vertices $u_3$, $x_2$, $y_2$. We start drawing remaining red $C_5$ from the vertex $z$. We can easily see that we cannot have edges $zu_1$ and $zu_2$ or $zx_1$ and $zx_2$ in this red $C_5$. Therefore we have to draw first red edge from $z$ to the vertex on black $C_5$ and the second red edge to the vertex on black $C_3$. Thus we have only one possibility to draw the remaining red cycle $C_5$. Therefore the packing of $C_3 \cup C_5$ is unique. 

Now we present plane and not planar packing of $\mathbf{C_3 \cup C_6}$ which could be extended to packings of $C_3 \cup C_p$, where $p$ is from the set $\{8, 10, 12, \ldots \}$. Two distinct packings of $C_3 \cup C_6$ are presented in Figure \ref{fig3}. 
\begin{figure}[h!]
\centering
\includegraphics[width=12.5cm]{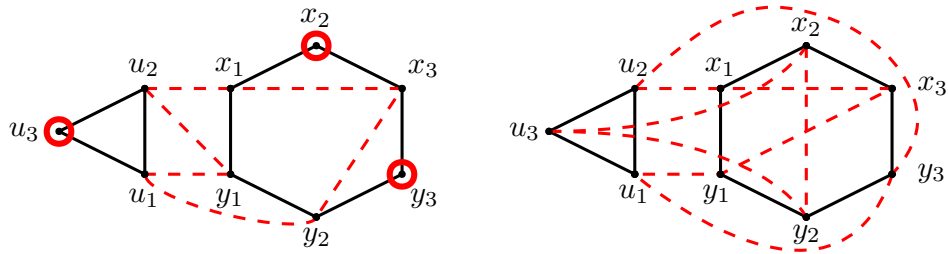}
\caption{Two distinct packings of $C_3 \cup C_6$. For the clarity of the drawing, the outer triangular face (connecting vertices marked with a circle) is not drawn.}
\label{fig3}
\end{figure}
We can easily see that the first (left) packing of $C_3 \cup C_6$ is plane. We extend this packing by replacing the edge $y_2y_3$ by the path $y_2a_1a_2\dots a_ly_3$ and the edge $x_3y_3$ by the path $x_3b_1b_2\dots b_ly_3$, where $l=\frac{p-6}{2}$. Then we add a path $y_2b_1a_1b_2a_2\dots b_la_l$ and replace the edge $u_1y_2$ by the edge $u_1a_l$. Note that the packing which we obtain is plane. Now we prove that the second (right) packing of $C_3 \cup C_6$ is not planar. Vertices $u_1$, $u_2$ and $u_3$ induce a triangle. If we add to this triangle the vertex $x_1$ together with  paths $x_1u_2$, $x_1y_1u_1$ and $x_1x_2u_3$ we obtain a subgraph homeomorphic to $K_4$. Thus if we add to this subgraph the vertex $y_3$ together with subsequent paths $y_3u_1$, $y_3u_2$, $y_3y_2u_3$ and $y_3x_3x_1$ we obtain a subgraph which is a subdivision of $K_5$. It follows from Kuratowski's theorem that the graph is not planar. We can extend this packing by replacing the edge $y_1y_2$ by the path $y_1a_1a_2\dots a_ly_2$ and the edge $x_3y_3$ by the path $x_3b_1b_2\dots b_ly_3$, where $l=\frac{p-6}{2}$. Then we add a path $y_1b_1a_1b_2a_2\dots b_la_l$ and replace the edge $u_1y_1$ by the edge $u_1a_l$. Note also that the packing which we obtain is not planar.

We show plane and not planar packing of $\mathbf{C_3 \cup C_7}$ which could be extended to packings of $C_3 \cup C_p$, where $p$ is from the set $\{9, 11, 13, \ldots \}$. Two distinct packings of $C_3 \cup C_7$ are presented in Figure \ref{fig4}. 
\begin{figure}[h!]
\centering
\includegraphics[width=13.5cm]{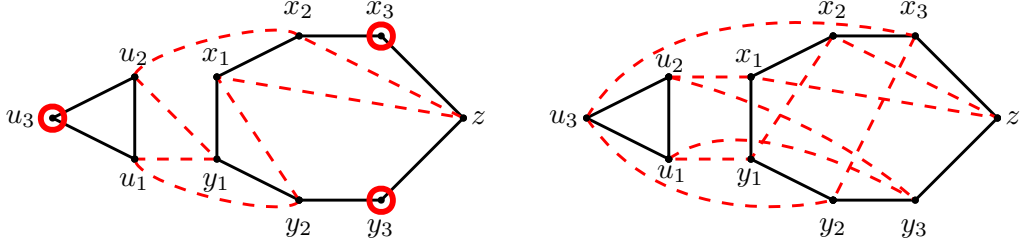}
\caption{Two distinct packings of $C_3 \cup C_7$. For the clarity of the drawing, the outer triangular face (connecting vertices marked with a circle) is not drawn.}
\label{fig4}
\end{figure}
We can easily see that the first (left) packing of $C_3 \cup C_7$ is plane. We extend this packing by replacing the edge $x_1x_2$ by the path $x_1a_1a_2\dots a_lx_2$ and the edge $zx_3$ by the path $zb_1b_2\dots b_lx_3$, where $l=\frac{p-7}{2}$. Then we replace the edge $zx_2$ by the edge $za_1$ and we add a path $a_1b_1a_2b_2\dots a_lb_lx_2$. Note that the packing which we obtain is plane. Now we prove that the second (right) packing of $C_3 \cup C_7$ is not planar. Vertices $u_1$, $u_2$ and $u_3$ induce a triangle. If we add to this triangle the vertex $x_1$ together with paths $x_1u_2$, $x_1y_1u_1$ and $x_1x_2x_3u_3$ we obtain a subgraph homeomorphic to $K_4$. Thus if we add to this subgraph the vertex $y_3$ together with subsequent paths $y_3u_1$, $y_3u_2$, $y_3y_2u_3$ and $y_3zx_1$  we obtain a subgraph which is a subdivision of $K_5$. It follows from Kuratowski's theorem that the graph is not planar. We can extend this packing by replacing the edge $y_3z$ by the path $y_3a_1a_2\dots a_lz$ and the edge $x_2x_3$ by the path $x_2b_1b_2\dots b_lx_3$, where $l=\frac{p-7}{2}$. Then we replace the edge $x_1z$ by the edge $x_1a_1$ and we add a path $a_1b_1a_2b_2\dots a_lb_lz$. Note also that the packing which we obtain is not planar.

\subsubsection{Subcase $G=C_4 \cup C_p$ where $p \geq 4$.} 
For a graph $G=\mathbf{C_4 \cup C_4}$ we show bipartite and not bipartite packing. 
In the graph $G$ we denote by $x_1, \dots, x_4$ and $y_1, \dots, y_4$ vertices from two sets $X$ and $Y$. In both packings we draw black cycles $B(\{x_1, x_2\}, \{y_1, y_2\})$ and  $B(\{x_3, x_4\}, \{y_3, y_4\})$. In the first packing of $G$ we draw red cycles $B(\{x_1, x_2\}, \{y_3, y_4\})$ and $B(\{x_3, x_4\}, \{y_1, y_2\})$. Note that $X$ and $Y$ are independent. In the second packing we draw red cycles $B(\{x_2, y_2\}, \{x_3, y_3\})$ and $B(\{x_1, y_1\}, \{x_4, y_4\})$. Note that each red cycle with two independent black edges induce a subgraph $K_4$. Therefore this packing of $G$ is not bipartite.

We present plane and not planar packing of $\mathbf{C_4 \cup C_5}$. These two packings could be extended to packings of $C_4 \cup C_p$, where $p$ is from the set $\{9, 11, 13 \ldots \}$. 
Two distinct packings of $C_4 \cup C_5$ are presented below in Figure \ref{fig6}. 
\begin{figure}[h!]
\centering
\includegraphics[width=11.4cm]{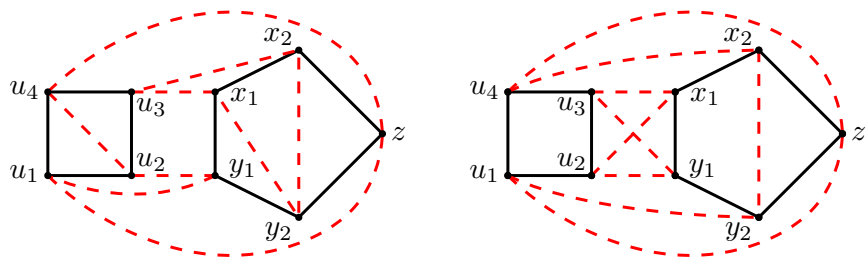}
\caption{Two distinct packings of $C_4 \cup C_5$.}
\label{fig6}
\end{figure}
It is obvious that the first (left) packing of $C_4 \cup C_5$ is plane. We extend this packing by replacing the edge $zy_2$ by the path $za_1a_2\dots a_ly_2$ and the edge $zx_2$ by the path $zb_1b_2\dots b_lx_2$, where $l=\frac{p-5}{2}$ and $l>1$. Then we replace the edge $u_1z$ by the edge $u_1a_1$ and we add a path $a_1b_1a_2b_2\dots a_lb_lz$. Note that the packing which we obtain is plane. Note also that the presented extension of plane packing does not work for $C_4 \cup C_7$. This case will be consider later. 

Now we prove that the second (right) packing of $C_4 \cup C_5$ is not planar and contains a subgraph $K_4$. Vertices $x_1$, $y_1$, $u_2$ and $u_3$ induce $K_4$. If we add to this subgraph the vertex $x_2$ together with paths $x_2x_1$, $x_2y_2y_1$, $x_2u_4u_3$ and $x_2zu_1u_2$ we obtain a subgraph which is a subdivision of $K_5$. It follows from Kuratowski's theorem that the graph is not planar. We can extend this packing by replacing the edge $y_1y_2$ by the path $y_1a_1a_2\dots a_ly_2$ and the edge $x_2z$ by the path $x_2b_1b_2\dots b_lz$, where $l=\frac{p-5}{2}$. Then we replace the edge $x_2y_2$ by the edge $x_2a_1$ and we add a path $a_1b_1a_2b_2\dots a_lb_ly_2$. Note also that the packing which we obtain is not planar.

We show plane and not planar packing of $\mathbf{C_4 \cup C_6}$ which could be extended to packings of $C_4 \cup C_p$, where $p$ is from the set $\{8, 10, 12, \ldots \}$. Two distinct packings of $C_4 \cup C_6$ are presented in Figure \ref{fig7}. 
\begin{figure}[h!]
\centering
\includegraphics[width=12.1cm]{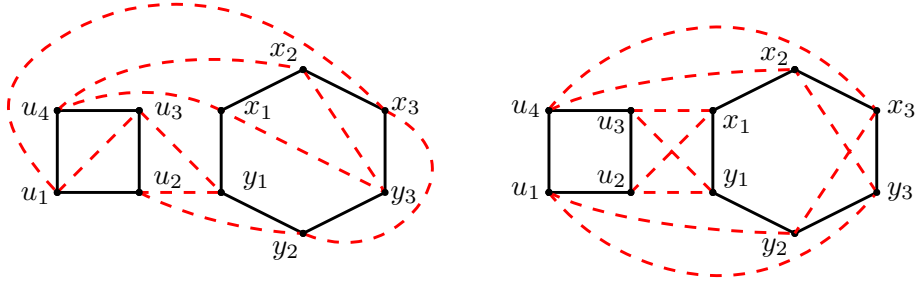}
\caption{Two distinct packings of $C_4 \cup C_6$.}
\label{fig7}
\end{figure}
It is obvious that the first (left) packing of $C_4 \cup C_6$ is plane. We extend this packing by replacing the edge $y_1x_1$ by the path $y_1a_1a_2\dots a_lx_1$ and the edge $y_2y_3$ by the path $y_2b_1b_2\dots b_ly_3$, where $l=\frac{p-6}{2}$. Then we replace the edge $y_2x_3$ by the edge $b_lx_3$ and we add a path $y_2a_1b_1a_2b_2\dots a_lb_lx_3$. Note that the packing which we obtain is plane. 

Now we prove that the second (right) packing of $C_4 \cup C_6$ is not planar. Vertices $x_1$, $y_1$, $u_2$ and $u_3$ induce $K_4$. If we add to this subgraph the vertex $y_3$ together with paths $y_3x_2x_1$, $y_3y_2y_1$, $y_3u_1u_2$ and $y_3x_3u_4u_3$ we obtain a subgraph which is a subdivision of $K_5$. It follows from Kuratowski's theorem that the graph is not planar. We can extend this packing by replacing the edge $y_1y_2$ by the path $y_1a_1a_2\dots a_ly_2$ and the edge $x_1x_2$ by the path $x_1b_1b_2\dots b_lx_2$, where $l=\frac{p-6}{2}$. Then we replace the edge $u_1y_2$ by the edge $u_1a_1$ and we add a path $a_1b_1a_2b_2\dots a_lb_ly_2$. Note also that the packing which we obtain is not planar.

Now, we show the remaining two distinct packings of $G=\mathbf{C_4 \cup C_7}$, the first with a subgraph $K_4$ and the second without $K_4$. The packing of $G$ with a subgraph $K_4$ we obtain from Lemma \ref{l7}. The packing of $G$ without $K_4$ is presented in Figure \ref{fig8}.
\begin{figure}[h!]
\centering
\includegraphics[width=6.1cm]{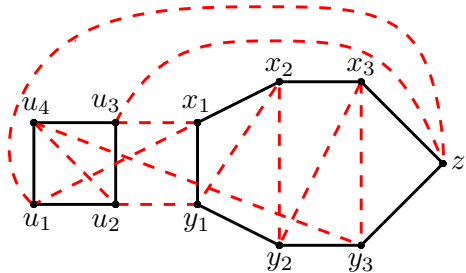}
\caption{The packing of $C_4 \cup C_7$ without $K_4$.}
\label{fig8}
\end{figure}
Note that a subgraph $K_4$ in a packing of $G$ could be obtained as a cycle $C_4$ from one copy of $G$ and two independent edges from the other copy or as three consecutive edges of a cycle $C_7$ from the first copy of $G$ and three consecutive edges of a cycle $C_7$ from the second copy. Therefore it suffices that we check vertices on cycles in black copy of $G$ whether they induce $K_4$. We left this easy check to the reader.

\subsection{Case k=3.} 

Let $G=C_{n_1} \cup C_{n_2} \cup C_{n_3}$, where $n_1\leq n_2 \leq n_3$. We can divide $G$ into two subgraphs $G_1=C_{n_1} \cup C_{n_2}$, $G_2=C_{n_3}$ and from the previous cases ($k=2$ and $k=1$) we get a packing of $G_1$ and $G_2$ except for $G=C_3 \cup C_4 \cup C_4$, $G=C_4 \cup C_4 \cup C_4$ and $G=C_3 \cup C_3 \cup C_p$ where $p \geq 3$. Thus from Lemma \ref{l1} the graph $G$ is not uniquely packable. Below we consider each of these exceptional graphs separately.

Two distinct packings of $\mathbf{C_3 \cup C_4 \cup C_4}$ are presented in Figure \ref{fig10}.
\begin{figure}[h!]
\centering
\includegraphics[width=12.5cm]{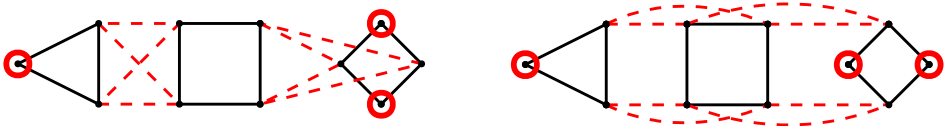}
\caption{Two distinct packings of $C_3 \cup C_4 \cup C_4$. For the clarity of the drawing, the cycles $C_3$ (connecting vertices marked with a circle) are not drawn.}
\label{fig10}
\end{figure}
We can easily see that the first (left) packing of $C_3 \cup C_4 \cup C_4$ contains a subgraph $K_4$. The second (right) packing of $G=C_3 \cup C_4 \cup C_4$ is without a subgraph $K_4$. Note that the subgraph $K_4$ in a packing of $G$ could be obtained as a cycle $C_4$ from one copy of $G$ and two independent edges from the other copy. Therefore it suffices that we check cycles $C_4$ form both copies of $G$.

For a graph $G=\mathbf{C_4 \cup C_4 \cup C_4}$ we show bipartite and not bipartite 
packing. In the graph $G$ we denote by $x_1, \dots, x_6$ and $y_1, \dots, y_6$ vertices from two sets $X$ and $Y$. In both packings of $G$ we draw black cycles $B(\{x_1, x_2\}, \{y_1, y_2\})$,  $B(\{x_3, x_4\}, \{y_3, y_4\})$  and $B(\{x_5, x_6\}, \{y_5, y_6\})$. In the first packing of $G$ we draw red cycles $B(\{x_1, x_2\}, \{y_3, y_4\})$, $B(\{x_3, x_4\}, \{y_5, y_6\})$ and $B(\{x_5, x_6\}, \{y_1, y_2\})$. Note that $X$ and $Y$ are independent. In the second packing we draw red cycles $B(\{x_2, y_2\}, \{x_3, y_3\})$, $B(\{x_4, y_4\}, \{x_5, y_5\})$ and $B(\{x_1, y_1\}, \{x_6, y_6\})$. Note that each red cycle with two independent black edges induce a subgraph $K_4$. Therefore this packing of $G$ is not bipartite.

Now we consider $G=\mathbf{C_3 \cup C_3 \cup C_p}$ where $p \geq 3$. If $p=3$ we start packing of $G$ from the black copy in which we denote by $x_1, x_2, x_3$, $y_1, y_2, y_3$ and $v_1, v_2, v_3$ vertices of three cycles $C_3$. Then we draw red triangles $x_1y_1v_1$, $x_2y_2v_2$ and $x_3y_3v_3$. We can easily see that this packing is unique. 

Two distinct packings of $G=C_3 \cup C_3 \cup C_4$ are presented in Figure \ref{fig12}.
\begin{figure}[h!]
\centering
\includegraphics[width=9.5cm]{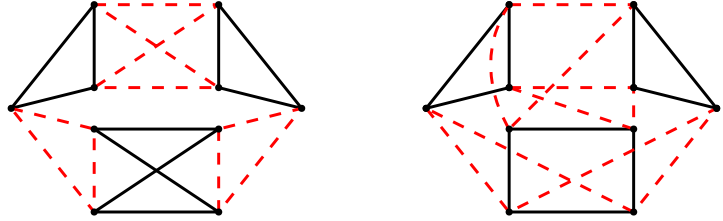}
\caption{Two distinct packings of $C_3 \cup C_3 \cup C_4$.}
\label{fig12}
\end{figure}
We can easily see that the first (left) packing of $G$ contains two subgraphs $K_4$. The second (right) packing of $G$ is without a subgraph $K_4$. Note that a subgraph $K_4$ in a packing of $G$ could be obtained as a cycle $C_4$ from one copy of $G$ and two independent edges from the other copy. Therefore it suffices that we check cycles $C_4$ form both copies of $G$.

Two distinct packings of $G=C_3 \cup C_3 \cup C_5$ are presented in Figure \ref{fig16}.
\begin{figure}[h!]
\centering
\includegraphics[width=11.7cm]{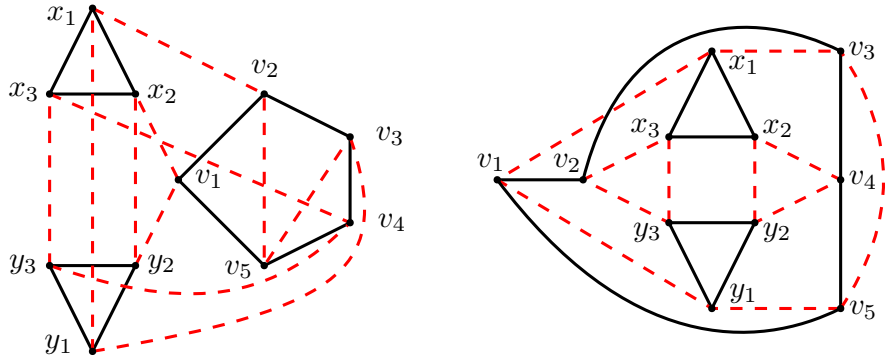}
\caption{Two distinct packings of $C_3 \cup C_3 \cup C_5$.}
\label{fig16}
\end{figure}
In the first (left) packing of $G$ there is a vertex $v_5$ such that its neighborhood induces a path of length three, while (as is relatively easy to check) the second (right) packing of $G$ does not contain such a vertex.

Let $G=C_3 \cup C_3 \cup C_6$. It is easy to see that we can pack two black triangles with a red cycle $C_6$ and two red triangles with a black cycle $C_6$. Thus we have disconnected packing of $G$. Then from Lemma \ref{l1} we get connected packing of $G$.

We present two distinct packings of $C_3 \cup C_3 \cup C_p$ where $p \geq 7$, the first with a subgraph $K_4$ and the second without $K_4$. We start from the first packing of $G$. We denote by $a_1, a_2, a_3, a_4$ four consecutive vertices of a cycle $C_p$ from $G$. Let $G'=G\setminus \{a_1, a_2, a_3, a_4\}$. We can easily see that $e(G') \leq |V(G')|-1$. By Theorem \ref{E3} there is a packing of $G'$. Thus from Lemma \ref{l7} we get a packing of $G$ with a subgraph $K_4$.

The second packing of $C_3 \cup C_3 \cup C_7$ and  $C_3 \cup C_3 \cup C_8$  without $K_4$ is presented in Figure \ref{fig13}. 
\begin{figure}[h!]
\centering
\includegraphics[width=\textwidth]{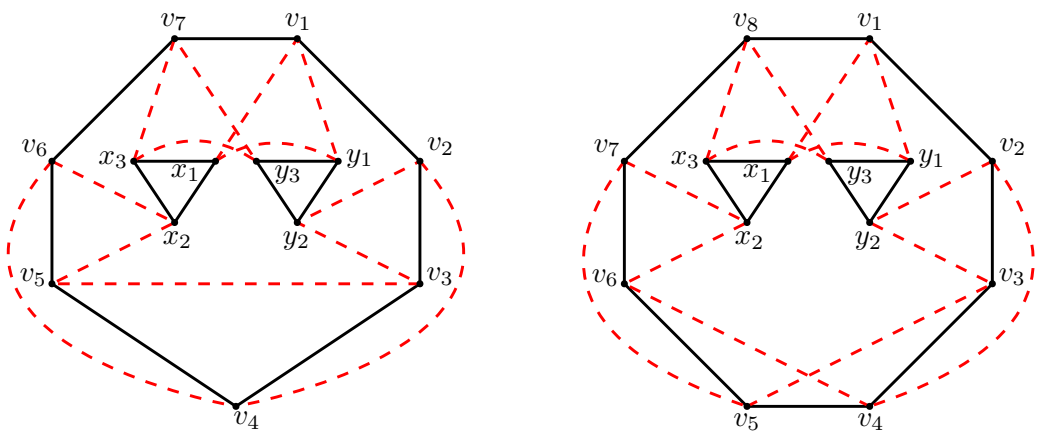}
\caption{The packing of $C_3 \cup C_3 \cup C_7$ and  $C_3 \cup C_3 \cup C_8$ without $K_4$.}
\label{fig13}
\end{figure}
Note that a subgraph $K_4$ in a packing of $G=C_3 \cup C_3 \cup C_p$, where $p\in \{7,8\}$, could be obtained as three consecutive edges of a cycle $C_p$ from the black copy of $G$ and three consecutive edges of a cycle $C_p$ from the red copy. 
Note that then red edges are of length two and three with respect to the distance on black cycle. Therefore it suffices that we check vertices on a cycle $C_p$ in one copy of $G$ whether they induce $K_4$. We left this easy check to the reader. 

We can extend the packing of $C_3 \cup C_3 \cup C_7$ to packings of  $C_3 \cup C_3 \cup C_p$, where $p$ is from the set $\{9, 11, 13, \ldots \}$. We replace the edge $v_3v_2$ by the path $v_3a_1a_2\dots a_lv_2$ and the edge $v_5v_6$ by the path $v_5b_1b_2\dots b_lv_6$, where $l=\frac{p-7}{2}$. Then we replace the edge $v_4v_2$ by the edge $v_4a_1$ and we add a path $a_1b_1a_2b_2\dots a_lb_lv_2$. Similarly we can extend the packing of $C_3 \cup C_3 \cup C_8$ to packings of  $C_3 \cup C_3 \cup C_p$, where $p$ is from the set $\{10, 12, 14, \ldots \}$. We replace the edge $v_3v_2$ by the path $v_3a_1a_2\dots a_lv_2$ and the edge $v_6v_7$ by the path $v_6b_1b_2\dots b_lv_7$, where $l=\frac{p-8}{2}$. Then we replace the edge $v_4v_2$ by the edge $v_4a_1$ and we add a path $a_1b_1a_2b_2\dots a_lb_lv_2$. Both  presented extensions of $C_3 \cup C_3 \cup C_7$ and $C_3 \cup C_3 \cup C_8$ do not contain a clique $K_4$, because we added red edges between vertices with distance more then three with respect to black cycle.

\section{Case $k\geq4$}
\subsection{Case k=4.}

Let $G=C_{n_1} \cup C_{n_2} \cup C_{n_3} \cup C_{n_4}$, where $n_1\leq n_2 \leq n_3 \leq n_4$. If at least two $n_i$ where $i \in \{1, 2, 3, 4\}$ are different from three then we can divide $G$ into two parts $G=G_1 \cup G_2$ so that $G_1$ and $G_2$ have packing. Therefore from Lemma \ref{l1} the graph $G$ is
not uniquely packable. Similarly when $n_4\geq5$. Thus we have to consider two subcases  $G=C_3 \cup C_3 \cup C_3 \cup C_3$ and $G=C_3 \cup C_3 \cup C_3 \cup C_4$. 

We start packing of $G=\mathbf{C_3 \cup C_3 \cup C_3 \cup C_3}$ from the black copy in which we denote vertices creating triangle $T_i$ by $\{a_i, b_i, c_i\}$ for $i\in \{1, 2, 3, 4\}$. Then we draw four red triangles with sets of vertices:  $\{a_1, a_2, a_3\}$, $\{b_2, b_3, b_4\}$, $\{c_1, c_3, c_4\}$ and $\{b_1, c_2, a_4\}$. Now we show that this packing of $G$ is unique (up to isomorphism). Note that all triangles in a packing of $G$ are "real" i.e. have all edges form black or red copy of $G$. We claim that each three black triangles include exactly one red triangle. We take three arbitrary black triangles $T_1$, $T_2$ and $T_3$. First, suppose that these black triangles do not include any red triangle. Thus each red triangle has at least one vertex outside $T_1 \cup T_2 \cup T_3$, namely in $T_4$. We get a contradiction. Second, suppose that these black triangles include exactly two triangles. Then the remaining two red triangles can use at most two vertices from $T_4$. We also get a contradiction. Therefore from the fact that three black triangles do not include three red triangles we get a confirmation of our claim. Thus without loss of generality we can assume that the second red triangle includes vertices from black triangles $T_2$, $T_3$ and $T_4$. This implies that the packing of remaining two red triangles is determined. Therefore the packing of $G$ is unique up to isomorphism.

Two distinct packings of $G=\mathbf{C_3 \cup C_3 \cup C_3 \cup C_4}$ are presented in Figure \ref{fig14}. 
\begin{figure}[h!]
\centering
\includegraphics[width=\textwidth]{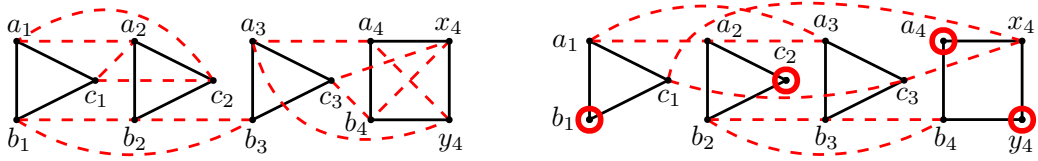}
\caption{Two distinct packings of $C_3 \cup C_3 \cup C_3 \cup C_4$.  For the clarity of the drawing, the cycle $C_4$ (connecting vertices marked with a circle) is not drawn.}
\label{fig14}
\end{figure}
We can easily see that the first (left) packing of $G$ is connected but the vertex $b_3$ is a cut vertex. Therefore it is not 2-connected. In the second (right) packing of $G$ each vertex from black triangle has two red edges to different black cycles. Moreover each two vertices from black triangle have edges to three remaining cycles. Therefore removing one vertex does not disconnect the graph.
\subsection{Case k=5.} 

Let $G=C_{n_1} \cup C_{n_2} \cup C_{n_3} \cup C_{n_4} \cup C_{n_5}$, where $n_1\leq n_2 \leq n_3 \leq n_4 \leq n_5$. If $n_5\geq4$ we can divide $G$ into two parts $G=G_1 \cup G_2$ so that $G_1=C_{n_1} \cup C_{n_2} \cup C_{n_3}$ and $G_2=C_{n_4} \cup C_{n_5}$ have packing. Therefore from Lemma \ref{l1} we know that the graph $G$ is not uniquely packable. Thus we have to consider $G=5C_3$.

We present two distinct packings of $G=5C_3$. We start both packings of $G$ from black copy in which we denote vertices creating triangle  $T_i$ by $\{a_i, b_i, c_i\}$ for $i\in \{1, \ldots, 5\}$. Then in the first packing of $G$ we draw five red triangles with sets of vertices:  $\{a_1, a_2, a_3\}$, $\{b_1, a_4, a_5\}$, $\{c_1, c_2, b_5\}$, $\{b_2, b_3, b_4\}$ and $\{c_3, c_4, c_5\}$. In the second packing of $G$ we draw five red triangles with sets of vertices: $\{a_1, a_2, a_3\}$, $\{b_1, b_2, b_3\}$, $\{c_1, b_4, b_5\}$, $\{c_2, a_4, a_5\}$ and $\{c_3, c_4, c_5\}$. Note that all triangles in both packings of $G$ are "real".
We can easily see that in the first packing of $G$ each nine vertices induce at most four triangles. In the second packing of $G$ nine vertices from black triangles $T_1$, $T_2$ and $T_3$ induce also two red triangles. Thus there exists nine vertices which induce five triangles in the second packing of $G$.

\subsection{Case $k\geq 6$.} 
Let $G=C_{n_1} \cup C_{n_2}\cup \ldots \cup C_{n_k}$, where $n_1\leq n_2 \leq \ldots \leq n_k$. We can divide $G$ into two parts $G=G_1 \cup G_2$ so that $G_1=C_{n_1} \cup C_{n_2} \cup C_{n_3}$ and $G_2=C_{n_4} \cup C_{n_5} \cup \ldots \cup C_{n_k}$. From the previous cases and the fact that $k\geq 6$ we have packings of $G_1$ and $G_2$. Thus $G$ has a disconnected packing. Therefore from Lemma \ref{l1} we get connected packing and we know that the graph $G$ is not uniquely packable.


\begin{thebibliography}{99}
\bibitem{BoEl}{\sc B.Bollob\'as and S.E.Eldridge}, Packings of graphs and applications to computational complexity, {\it  J. Combin. Theory} {\bf B 25} (1978), 105--124.

\bibitem{BuSc}{\sc  D.Burns and S.Schuster}, Every $(p,p-2)$ graph is contained in its complement, {\it J. Graph Theory} {\bf 1} (1977){\bf ,} 277--279.

\bibitem{BuSc1}{\sc  D.Burns and S.Schuster}, Embedding $(n,n-1)$ graphs   in   their complements, {\em Israel J. Math.\/} {\bf 30} (1978){\bf ,} 313--320.

\bibitem{Otwinowska}{\sc J.Otfinowska and M.Woźniak}, A Note on Uniquely Embeddable Forests,  {\em Discussiones Mathematicae-Graph Theory},  {\bf 33.1} (2013), 193--201.

\bibitem{SaSp} {\sc N. Sauer and J. Spencer}, Edge disjoint placement of graphs, {\em J. Combin. Theory\/}  Ser. B {\bf 25} (1978){\bf ,} 295--302.

\bibitem{Woz1}{\sc M.Wo\'zniak},  A note on uniquely embeddable graphs, {\em Discussiones Mathematicae-Graph Theory}, {\bf 18} (1998), 15--21.

\bibitem{Woz5}{\sc M.Wo\'zniak}, Packing of graphs,  {\em Dissertationes Mathematicae}, {\bf 362} (1997), pp.78.

\bibitem{Woz6}{\sc M.Wo\'zniak}, Packing of graphs and permutation -- a survery, {\em Discrete Math.}, {\bf 276} (2004), 379--391.

\bibitem{Yap}{\sc H.P.Yap}, Packing of graphs -- a survey,
{\em Discrete Math.\/} {\bf 72} (1988), 395--404.
\end{thebibliography}
\end{document}